\documentclass{amsart}

\usepackage{amsmath,amssymb,amsbsy,amsthm,amsfonts,mathtools,hyperref}
\usepackage{enumitem}
\usepackage{wrapfig}
\usepackage[top=3.5cm,bottom=3.5cm,left=3.7cm,right=3.7cm]{geometry}

\allowdisplaybreaks[3]

\begin{document}

\newcommand{\ddd}{\,{\rm d}}

\def\note#1{\marginpar{\small #1}}
\def\tens#1{\pmb{\mathsf{#1}}}
\def\vec#1{\boldsymbol{#1}}
\def\norm#1{\left|\!\left| #1 \right|\!\right|}
\def\fnorm#1{|\!| #1 |\!|}
\def\abs#1{\left| #1 \right|}
\def\ti{\text{I}}
\def\tii{\text{I\!I}}
\def\tiii{\text{I\!I\!I}}

\newcommand{\loc}{{\rm loc}}
\def\diver{\mathop{\mathrm{div}}\nolimits}
\def\grad{\mathop{\mathrm{grad}}\nolimits}
\def\Div{\mathop{\mathrm{Div}}\nolimits}
\def\Grad{\mathop{\mathrm{Grad}}\nolimits}
\def\tr{\mathop{\mathrm{tr}}\nolimits}
\def\cof{\mathop{\mathrm{cof}}\nolimits}
\def\det{\mathop{\mathrm{det}}\nolimits}
\def\lin{\mathop{\mathrm{span}}\nolimits}
\def\pr{\noindent \textbf{Proof: }}

\def\pp#1#2{\frac{\partial #1}{\partial #2}}
\def\dd#1#2{\frac{\d #1}{\d #2}}
\def\bA{\tens{A}}
\def\T{\mathcal{T}}
\def\R{\mathcal{R}}
\def\bx{\vec{x}}
\def\be{\vec{e}}
\def\bef{\vec{f}}
\def\bec{\vec{c}}
\def\bs{\vec{s}}
\def\ba{\vec{a}}
\def\bn{\vec{n}}
\def\bphi{\vec{\varphi}}
\def\btau{\vec{\tau}}
\def\bc{\vec{c}}
\def\bg{\vec{g}}
\def\mO{\mathcal{O}}
\def\pmO{\partial\mathcal{O}}
\def\bE{\tens{\varepsilon}}
\def\bsig{\tens{\sigma}}
\def\bW{\tens{W}}
\def\bT{\tens{T}}
\def\bxi{\tens{\xi}}
\def\bD{\tens{D}}
\def\bF{\tens{F}}
\def\bB{\tens{B}}
\def\bV{\tens{V}}
\def\bS{\tens{S}}
\def\bI{\tens{I}}
\def\bi{\vec{i}}
\def\bv{\vec{v}}
\def\bfi{\vec{\varphi}}
\def\bk{\vec{k}}
\def\b0{\vec{0}}
\def\bom{\vec{\omega}}
\def\bw{\vec{w}}
\def\p{\pi}
\def\bu{\vec{u}}
\def\bz{\vec{z}}
\def\bep{\vec{e}_{\textrm{p}}}
\def\dbep{\dot{\vec{e}}_{\textrm{p}}}
\def\bee{\vec{e}_{\textrm{el}}}
\def\dbee{\dot{\vec{e}}_{\textrm{el}}}
\def\ID{\mathcal{I}_{\bD}}
\def\IP{\mathcal{I}_{p}}
\def\Pn{(\mathcal{P})}
\def\Pe{(\mathcal{P}^{\eta})}
\def\Pee{(\mathcal{P}^{\varepsilon, \eta})}
\def\dbx{\,{\rm d} \bx}
\def\dx{\,{\rm d}x}
\def\dr{\,{\rm d}r}
\def\ds{\,{\rm d}s}
\def\dt{\,{\rm d}t}
\def\omer{\omega^{\delta}_{r_0}}
\def\ber{b^{\delta}_{r_0}}
\def\ver{v^{\delta,r_0}_{\bk, c}}
\def\hf{\hat{f}}


\newtheorem{Theorem}{Theorem}[section]

\newtheorem{Example}{Example}[section]
\newtheorem{Lemma}{Lemma}[section]
\newtheorem{Rem}{Remark}[section]
\newtheorem{Def}{Definition}[section]
\newtheorem{Col}{Corollary}[section]

\numberwithin{equation}{section}

\title[Minimizers for functionals with linear growth]{Globally Lipschitz minimizers for variational problems with linear growth}
\thanks{M.~Bul\'{\i}\v{c}ek and E.~Maringov\'{a} were supported by the ERC-CZ project LL1202  financed by the Ministry of Education, Youth and Sports, Czech Republic.  The support of the  the grant SVV-2016-260335 is also acknowledged. M.~Bul\'{\i}\v{c}ek is a  member of the Ne\v{c}as Center for Mathematical Modeling.}

\author[L. Beck]{Lisa Beck}
\address{Institut f\"{u}r Mathematik, Universit\"{a}t Augsburg,\\
 Universit\"{a}tsstr. 14, 86159~Augsburg, Germany}
\email{lisa.beck@math.uni-augsburg.de}

\author[M.~Bul\'{i}\v{c}ek]{Miroslav Bul\'i\v{c}ek}
\address{Mathematical Institute, Faculty of Mathematics and Physics, Charles University, Sokolovsk\'{a} 83, 186~75, Prague, Czech Republic}
\email{mbul8060@karlin.mff.cuni.cz}

\author[E.~Maringov\'{a}]{Erika Maringov\'{a}}
\address{Mathematical Institute, Faculty of Mathematics and Physics, Charles University, Sokolovsk\'{a} 83, 186~75, Prague, Czech Republic}
\email{maringova@karlin.mff.cuni.cz}

\keywords{Variational problems, linear growth, Lipschitz minimizers, non-convex domains}

\subjclass[2010]{35A01,35B65,35J70,49N60}

\begin{abstract}
We study the minimization of convex, variational integrals of linear growth among all functions in the Sobolev space $W^{1,1}$ with prescribed boundary values (or its equivalent formulation as a boundary value problem for a degenerately elliptic Euler--Lagrange equation). Due to insufficient compactness properties of these Dirichlet classes, the existence of solutions does not follow in a standard way by the direct method in the calculus of variations and in fact might fail, as it is well-known already for the non-parametric minimal surface problem. Assuming radial structure, we establish a necessary and sufficient condition on the integrand such that the Dirichlet problem is in general solvable, in the sense that a Lipschitz solution exists for any regular domain and all prescribed regular boundary values, via the  construction of appropriate barrier functions in the tradition of Serrin's paper~\cite{SERRIN69}.
\end{abstract}

\maketitle

\section{Introduction}

In this paper we are concerned with the existence of (unique) scalar-valued Lipschitz solutions to the Dirichlet problem
\begin{equation}\label{P1}
\begin{aligned}
-\diver \left(a(|\nabla u|)\nabla u \right)&=0 &&\textrm{ in } \Omega,\\
u&=u_0 &&\textrm{ on } \partial \Omega,
\end{aligned}
\end{equation}
where $\Omega \subset \mathbb{R}^d$ is a bounded, regular domain  and with regular prescribed boundary values~$u_0$. The focus is on coefficient functions $a \in \mathcal{C}^1(\mathbb{R}^+)$, which on the one hand represent a radial structure condition and which on the other belong to a convex linear growth problem which is naturally formulated in the Sobolev space $W^{1,1}(\Omega)$, meaning that we work under the permanent assumption that the function $s \mapsto a(s) s$ is increasing and remains bounded. In this setting, the existence of a weak solution to the Dirichlet problem~\eqref{P1} is equivalent to the existence of a minimizer of a related (convex) variational integral in the Dirichlet class $u_0 + W^{1,1}_0(\Omega)$, and we may equivalently look for a function $u \in u_0 + W^{1,1}_0(\Omega)$ such that for all smooth, compactly supported test function $\varphi \in \mathcal{D}(\Omega)$ we have
\begin{equation}\label{min}
\int_{\Omega}F(|\nabla u|)\dx \le \int_{\Omega}F(|\nabla u_0 +\nabla \varphi|)\dx,
\end{equation}
where $F$ and $a$ are linked via the identity
\begin{equation}
F'(s)=a(s)s \qquad \text{for all } s \in \mathbb{R}^+.
\label{min2}
\end{equation}

The minimal surface equation is clearly the most prominent example for such a Dirichlet problem, and other prototypic examples are given via the coefficient functions
\begin{equation}
\label{proto_example}
 a_p(s) \coloneqq (1+s^p)^{- \frac{1}{p}}
\end{equation}
for $s \in \mathbb{R}^+$ and $p>0$ (which for the specific case $p=2$ is just the minimal surface equation).

One peculiarity of linear growth problems is that, even if the equation~\eqref{P1} is monotone and the integrand $z \mapsto F(|z|)$ of the variational functional is convex, standard monotonicity methods and the direct method of the calculus of variations fail in general, since the Sobolev space $W^{1,1}(\Omega)$ is non-reflexive and hence has insufficient compactness properties. For the study of such linear growth problem of variational type, one common strategy is to extend in a first step the functional by lower semicontinuity (in the sense of Lebesgue and Serrin~\cite{LEBESGUE02,SERRIN61}) to the larger space $BV(\Omega)$ of functions of bounded variation, i.e., to consider, for fixed boundary values~$u_0$, the functional
\begin{equation*}
 w \mapsto \inf \Big\{ \liminf_{n \to \infty} \int_{\Omega}F(|\nabla w_n|)\dx \colon (w_n)_{n \in \mathbb{N}} \textrm{ in } u_0 + W^{1,1}_0(\Omega) \textrm{ with } w_n \to w \textrm{ in } L^1(\Omega) \Big\}
\end{equation*}
with $w \in BV(\Omega)$. This extension also allows for an integral representation, see e.g.~\cite{GIAMODSOU79}, which consists in the original functional evaluated for the absolutely continuous part of the measure derivative and penalization terms for non-vanishing singular measure derivative or non-attainment of the prescribed boundary values (note that the trace operator is in general not continuous with respect to weak-$*$ convergence in $BV(\Omega)$). In a second step, one can then study minimizers of the extended functional in $BV(\Omega)$ (which can be interpreted as generalized minimizers of the original functional), which exist as a consequence of the direct method, applied in $BV(\Omega)$ equipped with the weak-$*$-topology (the lower semicontinuity was established by Reshetnyak~\cite{RESHETNYAK68}).

Returning to the original question, one can then investigate whether or not these generalized minimizers belong to $u_0 + W^{1,1}(\Omega)$ (which amounts to excluding the singular measure derivative and to show attainment of the boundary values). However, this is in general not the case, meaning that in fact no minimizer in the space $u_0 + W^{1,1}(\Omega)$ might exist. This situation has indeed been studied in full detail for the minimal surface equation, and a by now classical result by Miranda~\cite{MIRANDA71} states that for any locally pseudoconvex domain~$\Omega$ and any continuous prescribed boundary values~$u_0$ a unique minimizer $u \in \mathcal{C}(\bar{\Omega}) \cap \mathcal{C}^2(\Omega)$ exists. Moreover, this result is sharp, and neither the convexity assumption nor the regularity assumption on $u_0$ can be considerably weakened in order to guarantee the existence of minimizers in $u_0 + W^{1,1}(\Omega)$.

Concerning higher regularity of such generalized minimizers, let us mention briefly that Bildhauer and Fuchs~\cite{BILFUC02b,BILDHAUER02,BILDHAUER03,BILDHAUERBUCH} have investigated, also for the vectorial case, the role of the so-called $\mu$-ellipticity condition, which quantifies the degeneration of the second order derivatives of the integrand (which in our case would basically mean $F''(s) \geq c s^{-\mu}$ for $s \geq 1$). Here, mild degeneration $\mu \in (1,3)$ (plus some additional assumptions, as radial structure of the integrand) still allows to prove that the generalized minimizers introduced before are actually of class $\mathcal{C}^{1}_\loc(\Omega)$ (see \cite[Theorem~2.7]{BILDHAUER02}, but also \cite[Theorem~B]{MARPAPI06} and \cite[Theorem~1.3]{BECSCH15}), while in the limit case with degeneration $\mu = 3$ (as for the area functional) it is still possible to show that they belong to $W^{1,1}(\Omega)$, with improved $L \log L$ integrability of the gradients (see \cite[Theorem~2.5]{BILDHAUER02} and \cite[Corollary~1.13]{BECSCH13}). However, these paper focus primarily on the regularity of generalized minimizers, and attainment of the prescribed boundary values~$u_0$ is in fact not expected in this general setting, as highlighted above.

In the present paper, we proceed with an alternative (and also very classical) strategy, which directly addresses the minimization problem (or equivalently Dirichlet problem), without passing through the relaxed formulation. Indeed, the goal is to characterize integrands~$F$ (or equivalently coefficient functions~$a$), in terms of properties of~$F$ only, such that the minimization or Dirichlet problem admits a solution in $u_0 + W^{1,1}(\Omega)$, for any (regular) domain~$\Omega$ and boundary values~$u_0$. In fact, necessary and sufficient conditions for the solvability of the Dirichlet problem in the planar case~$d=2$ for the second order elliptic equation
\begin{equation}
\label{eqn_Bernstein}
  {\mathcal A}(Du) \cdot \nabla^2 u = 0
\end{equation}
(with the convention ${\mathcal A}(z) \cdot \tilde{z} \coloneqq  \sum_{i,j=1}^d {\mathcal A}_{ij}(z) \tilde{z}_{ij}$) was already investigated by Bernstein~\cite{BERNSTEIN12}, in terms of the Bernstein genre~$g$ defined via the validity of
\begin{equation*}
 c |z|^{2-g} \leq \frac{{\mathcal A}(z) \cdot ( z \otimes z) }{\tr{{\mathcal A}(z)}} \leq C  |z|^{2-g} \qquad \textrm{for all } z \in \mathbb{R}^d \setminus B_1(\b0)
\end{equation*}
for some constants $c \leq C$ (if well-defined), and then generalized by Leray~\cite{LERAY39}. One can easily calculate that the Bernstein genre of the equation for the prototypic coefficients~$a_p$ from~\eqref{proto_example} is given by~$p$ (hence, the minimal surface equation is of Bernstein genre~$2$). Bernstein's discovery in~\cite{BERNSTEIN12} was that the question of the solvability of the Dirichlet problem splits into the two classes of genre $g \leq 1$ and $g > 1$. While for the first class the Dirichlet problem is in general solvable, one needs to impose curvature restriction on the second class (as the convexity conditioned mentioned before in Miranda's result). An extension to the higher-dimensional case $d \geq 2$ and a systematic treatment of more general non-uniformly elliptic equations of the form~\eqref{eqn_Bernstein} was given later by Serrin~\cite{SERRIN69}. He defined the equation to be \emph{regularly elliptic} if
\begin{equation}
\label{cond_reg_elliptic_1}
 \frac{{\mathcal A}(z) \cdot ( z \otimes z) }{\tr{{\mathcal A}(z)}} \geq \Phi(|z|)  \qquad \textrm{for all } z \in \mathbb{R}^d \setminus B_1(\b0)
\end{equation}
holds for some increasing function $\Phi \in \mathcal{C}(\mathbb{R}^+)$ satisfying
\begin{equation}
\label{cond_reg_elliptic_2}
  \int_{1}^\infty \Phi(t) t^{-2} \dt = \infty.
\end{equation}
Obviously, equations with a well-defined Bernstein genre~$g$ are regularly elliptic if and only if $g \leq1$. Furthermore, our specific equation~\eqref{P1} is regularly elliptic in particular if the left-hand side in~\eqref{cond_reg_elliptic_1} is increasing and if we have in addition
\begin{equation*}
 \int_1^\infty \frac{a'(t) t + a(t)}{a'(t)t+d a(t)} \dt = \int_1^\infty \frac{F''(t)}{a'(t)t+d a(t)} \dt=\infty.
\end{equation*}
The relevance of the structure condition~\eqref{cond_reg_elliptic_2} consists in the fact that it implies a priori estimates for the gradient of solutions on the boundary, via the construction of so-called \emph{global barriers functions}, and interior a priori bounds follow in turn. Since the existence of solutions to the Dirichlet problem can be reduced to the proof of a priori estimates, Serrin obtained as a consequence that regularly elliptic Dirichlet problems allow for a solution, for arbitrary (regular) domains and prescribed boundary values (while for non-regularly elliptic equations it is in general again necessary to impose restrictions on the domain, e.g. on the curvatures of the boundary~$\partial \Omega$).

The main result of the present article concerns the solvability of the Dirichlet problem in $u_0 + W^{1,1}_0(\Omega)$ and higher regularity of the solution in the sense of Lipschitz continuity, which we prove simultaneously, following the strategy of Serrin's work~\cite{SERRIN69}. We here work under a radial structure condition, which allows for an easier construction of barriers, and we work under a bounded oscillation assumption, which in some sense acts as a substitute for the monotonicity assumption of the function~$\Phi$ introduced above. We then obtain a necessary and sufficient condition for the solvability of the Dirichlet problem, in terms of an integral condition as in~\eqref{cond_reg_elliptic_2}, and the precise statement is the following:

\begin{Theorem}\label{T1}
Let $F\in \mathcal{C}^2(\mathbb{R}^+)$ be a strictly convex function with $\lim_{s \to 0} F'(s) = 0$ which satisfies, for some constants $C_1,C_2 >0$,
\begin{equation}
\begin{aligned}
C_1s -C_2 &\le F(s)\le C_2(1+s) &&\textrm{ for all } s\in \mathbb{R}^+,\\
\frac{F''(s)}{F''(t)}&\le C_2  &&\textrm{ for all } s\ge 1 \textrm{ and } t\in [s/2,2s].
\end{aligned}\label{A1}
\end{equation}
Then the following statements are equivalent:
\begin{enumerate}[font=\normalfont, label=(\roman{*}), ref=(\roman{*})]
 \item For arbitrary domains $\Omega$ of class $\mathcal{C}^{1}$ satisfying an exterior ball condition and arbitrary prescribed boundary values $u_0 \in \mathcal{C}^{1,1}(\overline{\Omega})$ there exists a unique function $u\in \mathcal{C}^{0,1}(\overline{\Omega})$ solving~\eqref{P1}.
 \item The function $F$ satisfies
\begin{equation}
\int_1^{\infty} t F''(t) \dt =\infty.\label{A2}
\end{equation}
\end{enumerate}
\end{Theorem}

The proof of Theorem~\ref{T1} will be divided into two parts. In Section~\ref{nonexistence} we first deal with the failure of the existence of Lipschitz solutions to some regular boundary value problem to~\eqref{P1} if the assumption~\eqref{A2} is not satisfied, by an adaptation of an example constructed by Finn~\cite{FINN65} for the minimal surface equation. We here emphasize that this non-existence result concerns the \emph{general} solvability of the Dirichlet problem, and in fact, some restricted solvability results, for specific (non-convex) domains and boundary values, might still be true even if~\eqref{A2} is violated, see for instance \cite[Theorem~2.1]{BULMALRAJWAL15}. The rest of the paper is then devoted to the proof of existence of Lipschitz solutions if assumption~\eqref{A2} holds. To this end, we provide in Section~\ref{Auxiliary} some auxiliary lemmata, before proceeding in Section~\ref{main_proof} to the main proof, which consists of a number of steps. First, since the existence of solutions cannot be addressed directly, we perform an approximation of the original Dirichlet problem by a family of Dirichlet problems exhibiting a quadratic growth condition (thus, admitting solutions in $u_0 + W^{1,2}_0(\Omega)$ by the direct method). Then, the passage to the limit yields a Lipschitz solution with boundary values~$u_0$ (and not only a generalized solution) if we can show uniform $W^{1,\infty}(\Omega)$ estimates for the solutions of the approximate problems. This is achieved by the construction of appropriate barrier functions, in the tradition of Serrin's paper~\cite{SERRIN69}, and concludes the proof of Theorem~\ref{T1}. Finally, it is worth to mention that the integral condition~\eqref{A2} is trivially satisfied (and hence that every Dirichlet problem is solvable) if~$F$ satisfies the aforementioned $\mu$-ellipticity condition studied by Bildhauer and Fuchs for some $\mu \in (1,2]$, thus providing a connection to the first approach to the minimization problem via relaxation. Moreover, we have recovered that the Dirichlet problem in the prototypic example with coefficients~\eqref{proto_example} is in general solvable for arbitrary regular domains and boundary values if and only if $p \leq 1$ holds.

\begin{Rem}
Concerning the assumptions on the domain and the function~$F$, let us note:
\begin{enumerate}[font=\normalfont, label=(\roman{*}), ref=(\roman{*})]
 \item A domain~$\Omega$ satisfies the exterior ball condition if there exists a number $r_0>0$ such that for every point~$\bx_0 \in \partial \Omega$ there is a ball $B_{r_0}(\tilde{\bx}_0)$ with $\overline{B_{r_0}(\tilde{\bx}_0)} \cap \overline{\Omega} = \{\bx_0\}$. Convexity or $\mathcal{C}^{1,1}$-regularity of the domain are sufficient for the exterior ball condition, see e.g.~\cite[Theorem 1.9.]{DALPHIN14}, thus, Theorem~\ref{T1} holds in particular for all convex domains of class $\mathcal{C}^{1}$ and for arbitrary domains of class $\mathcal{C}^{1,1}$.
 \item For simplicity, but without loss of generality, we can restrict ourselves in the whole paper to functions~$F$ which satisfy
 \begin{equation}
  F(0)=0 \quad \textrm{and} \quad \lim_{s\to \infty} \frac{F(s)}{s}=\lim_{s\to \infty}F'(s)=1, \label{A3}
 \end{equation}
 since the Dirichlet problem is invariant under addition and multiplication by a constant to $F$. Notice for the second relation that   the function~$s \mapsto F'(s)$ is monotonically increasing by convexity of~$F$ and therefore has a limit as $s\to \infty$. Moreover, strict positivity and finiteness of this limit follow from~\eqref{A1} and L'H\^{o}pital's rule, which also shows that the two limits in~\eqref{A3} coincide.
 \item With a change of variables \textup{(}and the normalization from the previous remark\textup{)}, we have
 \begin{equation*}
  \int_1^{\infty} t F''(t) \dt = \int_{F'(1)}^1 (F')^{-1}(s) \ds.
 \end{equation*}
 Since the latter integral can be rewritten via the conjugate function
 \begin{equation*}
 F^*(s^*) \coloneqq \sup_{s \in \mathbb{R}^+} \big\{s s^* - F(s)\big\} \qquad \textrm{for } s^* \in \mathbb{R}^+
 \end{equation*}
 \textup{(}which appears in the dual formulation of the Dirichlet problem in the sense of convex analysis\textup{)} as $\lim_{s^* \to 1} F^*(s^*) - F^*(F'(1))$, we observe that condition~\eqref{A2} is satisfied if and only if the conjugate function~$F^*$ explodes when approaching the boundary of its domain.
\end{enumerate}
\end{Rem}

Let us note that the result of Theorem~\ref{T1} could in fact be extended to less regular settings. One possibility is to consider convex functions~$F$ which are of class $\mathcal{C}^2$ only for large value. Since precisely only large gradients need to be avoided (uniformly) for the solutions to the approximative problems, such an asymptotic condition is in general sufficient (see e.g.~\cite[Theorem~1.2]{BECSCH15} for a related result). Moreover, if one is interested only in the existence of solutions in $u_0 + W^{1,1}_0(\Omega)$ (and not necessarily Lipschitz), one might work on domains which are only piece-wise of class~$\mathcal{C}^{1,1}$ and for more general (not Lipschitz) boundary values $u_0$.

We conclude the introduction with some comments on the notation used throughout the paper. For a set~$S$ in~$\mathbb{R}^d$ we write~$\partial S$ for its topological boundary and~$\overline{S}$ for its closure. Furthermore, for points in $\mathbb{R}^d$ we use bold letters (like $\bx$ and $\bx_0$), and the open ball in $\mathbb{R}^d$ wither center~$\bx_0$ and radius~$r$ is denoted by $B_r(\bx_0)$. Concerning function spaces, we work with the standard Lebesgue spaces~$L^p$ and Sobolev spaces~$W^{1,p}$, for $p \in [1,\infty]$, and we abbreviate the respective norms by $\| \cdot \|_{p}$ and $\| \cdot \|_{1,p}$, when the domain of reference is clear from the context.

\section{Non-existence of Lipschitz minimizers}
\label{nonexistence}

We first show the necessity of assumption~\eqref{A2} for the existence of Lipschitz minimizers to any regular boundary value problem to~\eqref{P1}. In what follows, we provide a simple example of a regular domain~$\Omega$ and regular boundary values~$u_0$ for which no Lipschitz solution (and in fact also no $W^{1,1}$ solution) to~\eqref{P1} exists. The construction is motivated by a well-known counterexample due to Finn~\cite{FINN65} for the minimal surface equation with $a_2(s)=(1+s^2)^{-\frac12}$, and we shall also here work on the annulus $\Omega\coloneqq  B_2(\b0)\setminus B_1(\b0)$ and with boundary values~$u_0$ which are constant on every connected component of~$\partial \Omega$, that is, we consider
\begin{equation}\label{P1ce}
\begin{aligned}
-\diver \Big( F'(|\nabla u|) \frac{\nabla u}{|\nabla u|} \Big) &= 0 \qquad &&\text{in}~ B_2(\b0) \setminus B_1(\b0), \\
 u&= 0 &&\text{on}~ \partial B_1(\b0), \\
 u&= M &&\text{on}~ \partial B_2(\b0),
\end{aligned}
\end{equation}
for some positive number $M\in \mathbb{R}^+$ (to be specified later). Moreover, we take $F \in \mathcal{C}^2(\mathbb{R}^+)$ as in Theorem~\ref{T1} satisfying
\begin{equation}
\int_1^{\infty} t F''(t) \dt = C_0 \label{not_A2}
\end{equation}
for some positive constant $C_0$ and, without loss of generality, also the normalization assumption~\eqref{A3}. Thanks to the strict convexity of~$F$, the monotonicity of~$F'$ and the radial symmetry of both the domain and the prescribed boundary values, the Lipschitz (or even $W^{1,1}$) solution to~\eqref{P1ce}, if it exists, is radially symmetric and can consequently be written as $u(\bx)=U(|\bx|)$ for all $\bx \in \Omega$ and some (Lipschitz) function $U \colon [1,2] \to \mathbb{R}$ with $U(1) = 0$ and $U(2)=M$. Thus, we also have $\nabla u(\bx) = U'(|\bx|) \frac{\bx}{|\bx|}$ for almost all $\bx \in \Omega$. In order to find a representation formula of $U$, we take an arbitrary function $\Phi \in \mathcal{D}([1,2])$ and extend it radially to a function $\varphi \in \mathcal{D}(\Omega)$ by setting $\varphi(\bx) \coloneqq \Phi(|\bx|)$. Then we test the weak formulation of~\eqref{P1ce} with $\varphi$ and find, by $\nabla \varphi(\bx) = \Phi'(|\bx|)\frac{\bx}{|\bx|}$, the transformation to polar coordinates and the radial symmetry of both functions~$U$ and~$\Phi$, the identity
\begin{align*}
0 &=\int_{\Omega} F'(|\nabla u(\bx)|) \frac{\nabla u(\bx)}{|\nabla u(\bx)|} \cdot \nabla \varphi(\bx) \dbx \\
  &=\int_{\Omega} F'(|U'(|\bx|)|) \frac{U'(|\bx|)}{|U'(|\bx|)|} \Phi'(|\bx|) \dbx \\
  &= d \omega_d \int_1^2 r^{d-1} F'(|U'(r)|) \frac{U'(r)}{|U'(r)|} \Phi'(r) \dr ,
\end{align*}
with $\omega_d$ the Lebesgue measure of the unit ball in $\mathbb{R}^d$ (and hence $d \omega_d$ the $(d-1)$-dimensional Hausdorff measure of the unit sphere in $\mathbb{R}^d$). Since $\Phi \in \mathcal{D}([1,2])$ was arbitrary, we deduce in a first step
\begin{equation*}
F'(|U'(r)|) \frac{U'(r)}{|U'(r)|} = \dfrac{c}{r^{d-1}}
\end{equation*}
for some constant $c$ and all $r \in [1,2]$. By assumptions on $F$, we next observe that no sign change of $U'$ may occur, hence, $U'$ is positive everywhere in $[1,2]$ and we have indeed
\begin{equation*}
1 > F'(U'(r)) = \dfrac{c}{r^{d-1}},
\end{equation*}
hence also $c \in (0,1)$. After inverting the last identity we can integrate (keeping in mind the boundary condition $U(1) = 0$) and find the desired representation formula
\begin{equation*}
U(r) = \int_1^{r} (F')^{-1} \Big(\dfrac{c}{s^{d-1}}\Big) \ds
\end{equation*}
for all $r \in [1,2]$. In turn, with the substitution $c/s^{d-1}= z$, that is, $s=c^{\frac{1}{d-1}}z^{\frac{1}{1-d}}$, we obtain the following upper bound on~$U$:
\begin{align*}
U(r) &= \int_c^{c/r^{d-1}} \frac{c^{\frac{1}{d-1}}}{1-d} z^{\frac{d}{1-d}} (F')^{-1}(z) \, \text{d}z \\
&= \frac{c^{\frac{1}{d-1}}}{d-1} \int_{c/r^{d-1}}^c  z^{\frac{d}{1-d}} (F')^{-1}(z) \, \text{d}z \\
&\le \frac{2^d}{c(d-1)} \int_{c/r^{d-1}}^c (F')^{-1}(z) \, \text{d}z.
\end{align*}
This provides indeed a nontrivial upper bound, as can be seen by a case distinction between small and large values of~$c$, in order to estimate the integral appearing on the right-hand side. In the case $0 < c \leq F'(1) <1$, we find by monotonicity of~$(F')^{-1}$
\begin{equation*}
 \int_{c/r^{d-1}}^c (F')^{-1}(z) \, \text{d}z \leq c (1 - r^{1-d}) < c,
\end{equation*}
while in the opposite case $0 < F'(1) < c < 1$ the change of variables with $z=F'(t)$ combined with~\eqref{not_A2} yields
\begin{align*}
\int_{c/r^{d-1}}^c (F')^{-1}(z) \, \text{d}z & = \int_{(F')^{-1}(c/r^{d-1})}^{(F')^{-1}(c)} t F''(t) \dt \\
  & \leq \int_0^{\infty}  t F''(t) \dt \leq F'(1) + C_0 .
\end{align*}
In conclusion, for all $r\in [1,2]$, we have derived the explicit upper bound
\begin{equation*}
U(r)\le \frac{2^d}{d-1} \Big( 1 +  \frac{C_0}{F'(1)}\Big),
\end{equation*}
which is a contradiction to $U(2)=M$ for any given $M\in \mathbb{R}^+$. Thus, we have just proven that if~\eqref{A2} does not hold, then we can find a smooth domain~$\Omega$ and smooth boundary values~$u_0$ such that all assumptions of Theorem~\eqref{T1} are satisfied, and such that no Lipschitz solution to problem~\eqref{P1} exists. \qed

\begin{Rem}
More subtle non-existence results can be obtained, similarly as in the case of the minimal surface equation. In particular, the class of regular domains which allow for non-existence results can be investigated. Based on the previous construction and a comparison principle, one can in fact show that for every non-pseudoconvex  regular domain~$\Omega$ there exists smooth prescribed boundary values such that no Lipschitz solution to the Dirichlet problem~\eqref{P1} exists.
\end{Rem}

\section{Auxiliary lemmata}
\label{Auxiliary}

In order to proceed to the proof of the second implication of Theorem~\ref{T1}, we first derive some auxiliary algebraic inequalities.

\begin{Lemma}\label{L1}
Let $F\in \mathcal{C}^2(\mathbb{R}^+)$ be a strictly convex function with $\lim_{s \to 0} F'(s) = 0$ which satisfies~\eqref{A1},~\eqref{A2} and~\eqref{A3}, and let $a \in  \mathcal{C}^1(\mathbb{R}^+)$ be given by~\eqref{min2}. Then, there hold
\begin{align}
 C_1s - C_2 \leq sF'(s) & \leq s \label{R1} \qquad \text{for every $s > 0$}, \\
\lim_{s\to \infty} sF''(s) & = 0,\label{R2}\\
\lim_{s\to \infty} s^2 a'(s) & = -1.\label{R3}
\end{align}
\end{Lemma}

\begin{proof}
We start by observing some simple consequences of the strict convexity of~$F$. As~$F'$ is monotonically increasing, assumption~\eqref{A1} gives
\begin{equation*}
C_1s - C_2 \le F(s) = \int_0^s F'(r)\dr \le \int_0^s F'(s)\dr =sF'(s),
\end{equation*}
which is the first inequality in~\eqref{R1}. We further observe that the assumptions $\lim_{s \to 0} F'(s) = 0$ and $\lim_{s \to \infty} F'(s) = 1$ yield immediately $0<F'(s)<1$ for all $s \in (0,\infty)$, which implies the second inequality in~\eqref{R1}, and moreover, $F''\in L^1(0,\infty)$ holds. Due to the integrability of $F''$, we next deduce the identity
\begin{equation}
  a(s)s=F'(s) = 1-\int_s^{\infty} F''(t)\dt, \label{pomer0}
\end{equation}
and since by the bounded oscillation assumption in~\eqref{A1} we have in particular
\begin{equation*}
 F''(s) s \leq C_2 \int_s^{2s} F''(t) \dt
\end{equation*}
for all $s>0$, also the claim~\eqref{R2} follows. Finally, differentiating~\eqref{pomer0} we find
\begin{equation}\label{pomer}
a'(s)s + a(s) = F''(s),
\end{equation}
and therefore, thanks to~\eqref{R2} and~\eqref{A3}, we obtain
$$
\lim_{s \to \infty} s^2 a'(s) = \lim_{s \to \infty} \left(s F''(s) - F'(s)\right) = -1,
$$
which is the last claim~\eqref{R3}.
\end{proof}

Secondly, we define the integrand for a comparison functional, which will be used later for the construction of appropriate barrier functions. To this end, we will essentially decrease the convexity for large values of the original integrand~$F$ (note that the properties of the integrand for large values are the most crucial ones), which from a heuristic point of view will make it harder to construct solutions (cp.~the calculations for annular domains in Section~\ref{nonexistence}). However, it turns out that as long as the fundamental condition~\eqref{A2} for the new integrand is satisfied, this construction is still possible, and it is precisely the weaker convexity for large values which will allow for the verification of the barrier condition.

\begin{Lemma}\label{L2}
Let $F\in \mathcal{C}^2(\mathbb{R}^+)$ be a strictly convex function with $\lim_{s \to 0} F'(s) = 0$ which satisfies~\eqref{A1},~\eqref{A2} and~\eqref{A3}. Then there exists a strictly positive, decreasing function $g\in \mathcal{C}(\mathbb{R}^+_0)$ with $\lim_{s \to \infty} g(s) = 0$ such that
\begin{equation}
\int_0^{\infty} t F''(t) g(t)\dt = \infty.\label{G1}
\end{equation}
Moreover, the function~$F_g$ defined via
\begin{equation}
F_g(s) \coloneqq \int_0^s \Big(1-\int_r^{\infty}F''(t)g(t)\dt \Big) \dr \label{G2}
\end{equation}
is strictly convex and satisfies $\lim_{s \to 0} F_g'(s) = 0$ and~\eqref{A1}--\eqref{A3}, with possibly different constants~$C_1$ and~$C_2$.
\end{Lemma}

\begin{proof}
We start by defining
\begin{equation}
\tilde{g}(s) \coloneqq \frac{1}{1+\int_0^s t F''(t) \dt},  \label{tilde}
\end{equation}
and we note that $\tilde{g}\le 1$ is a continuous, strictly decreasing function fulfilling $\tilde{g}(s)\to 0$ as $s\to \infty$. Next, we define
$$
A\coloneqq \int_0^{\infty} F''(t)\tilde{g}(t) \dt,
$$
and due to the properties of $F$ (more precisely, $F'(0) = 0$ and $\lim_{s \to \infty} F'(s)=1$) combined with $0 < \tilde{g}(s) \le 1$, we observe $A\in (0,1]$. Finally, we define
\begin{equation}
g(s)\coloneqq  \frac{\tilde{g}(s)}{A},\label{gdf}
\end{equation}
and we now show that all statements of the lemma are indeed fulfilled with such a choice of~$g$. Let us start with~\eqref{G1}. By the definition of~$g$ it directly follows that for any $s\ge 0$ we have
$$
\begin{aligned}
\int_0^s t F''(t)g(t) \dt &= \frac{1}{A} \int_0^s \frac{t F''(t)}{1+\int_0^t rF''(r)\dr} \dt\\
 & =\frac{1}{A} \int_0^s \frac{d}{\dt} \ln \Big(1+\int_0^t rF''(r)\dr \Big) \dt \\
 & = \frac{1}{A} \ln \Big(1+\int_0^s rF''(r)\dr \Big).
\end{aligned}
$$
Since $F''$ and $g$ are positive and due to the assumption~\eqref{A2}, we have that
$$
\begin{aligned}
\int_0^{\infty} t F''(t) g(t) \dt &= \lim_{s\to \infty}\int_0^s t F''(t)g(t) \dt \\
 & = \frac{1}{A} \lim_{s \to \infty}\ln \Big(1+\int_0^s rF''(r)\dr \Big) =\infty.
\end{aligned}
$$
Hence,~\eqref{G1} holds. Next, we show the properties of the function~$F_g$. By its definition~\eqref{G2} we immediately obtain $F_g(0)=0$. Moreover, we calculate its derivatives
\begin{equation}\label{hlp2}
F'_g(s)=1-\int_s^{\infty}F''(t)g(t)\dt, \qquad F''_g(s)= F''(s)g(s) >0,
\end{equation}
which shows the strict convexity and, via~\eqref{G1}, the validity of~\eqref{A2} for~$F_g$. Since by the definition of $g$ we have
\begin{equation*}
\int_0^{\infty}F''(t)g(t)\dt=1,
\end{equation*}
we find $\lim_{s \to 0} F_g'(s) = 0$ and $\lim_{s \to \infty} F'_g(s) = 1$, thus also~\eqref{A3} is satisfied for~$F_g$. Finally, concerning~\eqref{A1}, we note that the linear growth assumption follows immediately from the properties of $F_g'$, while the oscillation assumption of $F$ carries directly over to~$F_g$ (with constant $4 C_2^2$ instead of~$C_2$). This concludes the proof of the lemma.
\end{proof}

\section{Proof of  Theorem~\ref{T1}}
\label{main_proof}

We now prove the second (and main) implication of Theorem~\ref{T1}. To this end, we perform an approximation procedure, and in what follows, we will prove uniform estimates for the minimizers for a proper choice of approximative problems. This is indeed sufficient to recover the claim with the passage to the limit since the strict convexity of our functional implies uniqueness of minimizers (if it exists in the desired Dirichlet class at all). Thus, for arbitrary $\varepsilon>0$, we introduce the approximate functionals
\begin{equation*}
 w \mapsto \frac{\varepsilon}{2} \int_\Omega |\nabla w|^2 \dx + \int_\Omega F(|\nabla w|) \dx
\end{equation*}
and look for minimizers $u_\varepsilon$ in the Dirichlet class $u_0 + W^{1,2}_0(\Omega)$, which is equivalent to looking for weak solutions to the following approximate Dirichlet problem to~\eqref{P1}
\begin{equation}\label{P1e}
\begin{aligned}
-\varepsilon \Delta u_\varepsilon -\diver \big(a(|\nabla u_\varepsilon|)\nabla u_\varepsilon \big) &=0 &&\textrm{ in } \Omega,\\
u_\varepsilon &=u_0 &&\textrm{ on } \partial \Omega.
\end{aligned}
\end{equation}
Note that for the rest of the paper, we suppose, without explicit mentioning, that the function~$F$ satisfies~\eqref{A1},~\eqref{A2} and~\eqref{A3}, and that $a$ is related to $F$ via~\eqref{min2}.

Due to the application of the direct method of the calculus of variations to the approximate functionals (or the theory of monotone operators to the approximate Dirichlet problems), we have the existence of a unique weak solution $u_\varepsilon \in u_0 + W^{1,2}_0(\Omega)$. In addition, in view of the regularity of the prescribed boundary values, we have $u_\varepsilon \in \mathcal{C}^{1,\alpha}(\overline{\Omega})$ for some $\alpha >0$, see e.g.~\cite{GIAGIU84b}, and by difference quotient techniques (and possibly after regularization of~$F$ via an $\varepsilon$-mollifying kernel) we also have $u_\varepsilon \in W^{2,2}_\loc(\Omega)$. Our main goal is to show that the following uniform estimate holds
\begin{equation}
\|\nabla u_\varepsilon \|_{\infty}\le C(\Omega,F,u_0)\label{basis}
\end{equation}
with some constant $C(\Omega,F,u_0)$ being independent of $\varepsilon$. Indeed, having~\eqref{basis} in hands, we first find a subsequence converging weakly-$*$ to a function $u \in u_0 + W^{1,2}_0(\Omega)$. Then, when passing to the limit $\varepsilon \to 0$ in the approximate functionals (by lower semicontinuity) or in the approximate Dirichlet problems~\eqref{P1e} (by theory of monotone operators), the limit~$u \in u_0 + W^{1,\infty}_0(\Omega)$ turns out to be the desired solution. As is it actually Lipschitz regular, Theorem~\ref{T1} is therefore proven, provided that we can show that~\eqref{basis} holds.

\subsection{Reduction to the boundary estimates}\label{ss31}

We first show that, in order to have~\eqref{basis}, it is actually sufficient to control the normal derivatives of the solutions~$u_\varepsilon$ to the approximate problems uniformly on the boundary~$\partial \Omega$. To this end, we start by deriving some standard uniform estimates and denote by $C$ an universal constant depending only on $F$, $u_0$ and $\Omega$, but not on~$\varepsilon$. For simpler notation, we shall drop from now on the index~$\varepsilon$ and write~$u$ instead of~$u_\varepsilon$. Testing the weak formulation to~\eqref{P1e} with the function $u-u_0 \in W^{1,2}_0(\Omega)$, keeping in mind relation~\eqref{min2} and applying H\"{o}lder's inequality, we obtain
$$
\varepsilon \|\nabla u\|_2^2 + \int_{\Omega}F'(|\nabla u|)|\nabla u|\dx \le \varepsilon \|\nabla u\|_2 \|\nabla u_0\|_2 + \int_{\Omega} F'(|\nabla u|)|\nabla u_0|\dx.
$$
Hence, using Young's inequality and~\eqref{R1}, we deduce
\begin{equation}
\varepsilon \|\nabla u\|_2^2 + \|\nabla u\|_1 \le C.\label{AE1}
\end{equation}
Similarly, testing~\eqref{P1e} with the functions $(u\mp\|u_0\|_{\infty})_{\pm}$ (note that these functions are admissible since $(u\mp\|u_0\|_{\infty})_{\pm}=0$ holds on $\partial \Omega$), we get
\begin{equation*}
\varepsilon \int_{\Omega} |\nabla (u\mp\|u_0\|_{\infty})_{\pm}|^2 \dx + \int_{\Omega} a(|\nabla u|) |\nabla (u\mp\|u_0\|_{\infty})_{\pm}|^2 \dx =0.
\end{equation*}
Thus, it follows that
\begin{equation}
\|u\|_{\infty} \le \|u_0\|_{\infty} \le C. \label{AE2}
\end{equation}
To proceed further, we identify the equation for $|\nabla u|$. Applying $\frac{\partial}{\partial x_k}=:D_k$ to~\eqref{P1e}, multiplying the result by $D_k u$ and summing over $k=1, \ldots, d$, we obtain
\begin{align*}
0 & =-\varepsilon \sum_{k=1}^d D_ku \Delta D_k u - \sum_{k,i=1}^d D_k u D_i D_k \Big(F'(|\nabla u|)\frac{D_i u}{|\nabla u|}\Big)\\
  & =-\frac{\varepsilon}{2} \Delta |\nabla u|^2 + \varepsilon |\nabla^2 u|^2 -  \sum_{k,i=1}^d  D_{i} \Big(D_k \Big(F'(|\nabla u|)\frac{D_i u}{|\nabla u|} \Big)D_ku \Big) \\
  & \qquad + \sum_{k,i=1}^d D_{ik} u D_{k} \Big(F'(|\nabla u|)\frac{D_i u}{|\nabla u|}\Big)\\
  & = - \frac{\varepsilon}{2} \Delta |\nabla u|^2 -  \sum_{k,i=1}^d  D_{i} \big(A_{ik}(\nabla u) D_k |\nabla u| \big)\\
  & \quad + \varepsilon |\nabla^2 u|^2 + F''(|\nabla u|)|\nabla |\nabla u||^2  +F'(|\nabla u|)\frac{|\nabla^2 u|^2-|\nabla |\nabla u||^2}{|\nabla u|},
\end{align*}
where
\begin{equation*}
A_{ik}(\nabla u)\coloneqq \Big(|\nabla u| F''(|\nabla u|)\frac{D_i u D_ku }{|\nabla u|^2}+F'(|\nabla u|) \delta_{ik}-F'(|\nabla u|)\frac{D_i u D_ku }{|\nabla u|^2}\Big).
\end{equation*}
Consequently,
\begin{equation*}
-\frac{\varepsilon}{2} \Delta |\nabla u|^2 -  \sum_{k,i=1}^d  D_{i} \left(A_{ik}(\nabla u) D_k |\nabla u|\right)\le 0.
\end{equation*}
Since~$A$ is positively definite, we see that~$|\nabla u|^2$ is a sub-solution to a linear elliptic equation and therefore satisfies the minimum principle, i.e.,
\begin{equation*}
\|\nabla u\|_{\infty}\le \|\nabla u\|_{L^{\infty}(\partial \Omega)}.
\end{equation*}
In addition, since $u=u_0$ on~$\partial \Omega$, this implies
\begin{equation*}
\|\nabla u\|_{\infty}\le \|\nabla u_0\|_{\infty} + \Big\|\frac{\partial u}{\partial \bn}\Big\|_{L^{\infty}(\partial \Omega)},
\end{equation*}
where $\frac{\partial u}{\partial \bn}$ denotes the normal derivative of~$u$ on~$\partial \Omega$. Thus, in order to check~\eqref{basis} it remains to show that
\begin{equation}
\Big\|\frac{\partial u}{\partial \bn} \Big\|_{L^\infty(\partial \Omega)}\le C(\Omega,F,u_0).\label{basis2}
\end{equation}
The rest of the paper is devoted to the proof of~\eqref{basis2}, which will be shown via the barrier function technique.

\subsection{Prototype barrier function}
From now on, we fix the functions $g$ and $F_g$ according to Lemma~\ref{L2} and define
\begin{equation}\label{DA}
a_g(s)\coloneqq \frac{F'_g(s)}{s}.
\end{equation}
Clearly, all statements of Lemma~\ref{L1} hold also for $F_g$ and $a_g$ with possibly different constants $C_1,C_2>0$. Moreover, $F'_g$ is a strictly monotonically increasing mapping from $[0,\infty)$ to $[0,1)$ with continuous inverse. With the help of~$F_g$ we now define our prototype barrier function.

Let $r_0>0$ and $\delta\in (0,1)$ be arbitrary. We set for all $r\ge r_0$
\begin{equation}\label{DFb}
\ber(r)\coloneqq (F'_g)^{-1} \bigg(\frac{(1-\delta)^{d-1}r^{d-1}_0}{r^{d-1}}\bigg).
\end{equation}
It can be easily seen that $\ber \in \mathcal{C}^1[r_0,\infty)$ is a non-negative decreasing function. Finally, for all $x\in \mathbb{R}^d\setminus B_{r_0}(\b0)$, we define
\begin{equation}
\omer(\bx)\coloneqq \int_{r_0}^{|\bx|} \ber(r)\dr. \label{dfomer}
\end{equation}
By construction,~$\omer$ is a minimizer of the functional with integrand $F_g$ and equivalently a solution to the associated Dirichlet problem on the set $\mathbb{R}^d \setminus \overline{B_{r_{0}}(\b0)}$, but moreover, it also turns out to be super-harmonic on a subset of it.

\begin{Lemma}\label{L3}
For every $r_0>0$ and $\delta\in (0,1)$ the function~$\omer$ defined in~\eqref{dfomer} satisfies
\begin{equation}\label{P1v}
\begin{aligned}
-\diver \left( a_g(|\nabla \omer|)\nabla \omer \right) &=0 &&\textrm{ in } \mathbb{R}^d \setminus \overline{B_{r_{0}}(\b0)},\\
\omer &=0 &&\textrm{ on } \partial B_{r_{0}}(\b0).
\end{aligned}
\end{equation}
Furthermore, there holds
\begin{equation}\label{subhar}
-\Delta \omer (\bx) \ge 0 \qquad \textrm{for all } \bx \in \mathbb{R}^d \setminus \overline{B_{r_{0}}(\b0)} \textrm{ such that }a'_g(\ber(|\bx|))\le 0.
\end{equation}
\end{Lemma}

\begin{proof}
Using the definition of~$\omer$, we immediately see that~$\omer$ vanishes on $\partial B_{r_{0}}(\b0)$, and we further observe
\begin{equation}
\label{omerp}
\nabla \omer(\bx) = \ber(|\bx|)\frac{\bx}{|\bx|} \quad \text{and} \quad |\nabla \omer(\bx)| = \ber(|\bx|).
\end{equation}
Via the definition of $\ber$, we thus have
\begin{equation*}
F'_g(|\nabla \omer(\bx)|) = F'_g(\ber(|\bx|))= \frac{(1-\delta)^{d-1}r^{d-1}_0}{|\bx|^{d-1}}.
\end{equation*}
Consequently, for all $|\bx|>r_0$ there holds
\begin{equation}
\begin{aligned}
\diver \left(a_g(|\nabla \omer(\bx)|)\nabla \omer (\bx)\right)&=\diver\Big(F_g'(|\nabla \omer(\bx)|)\frac{\nabla \omer(\bx)}{|\nabla \omer (\bx)|}\Big) \\
  &=(1-\delta)^{d-1}r^{d-1}_0\diver \frac{\bx }{|\bx|^{d}}=0
\end{aligned}
\end{equation}
and the solution property~\eqref{P1v} follows. Finally, we check the super-harmonicity property of~$\omer$. In view of~\eqref{P1v} and~\eqref{omerp} we get
$$
\begin{aligned}
0&=-\diver \left( a_g(|\nabla \omer(\bx)|)\nabla \omer(\bx) \right)\\
&= - a_g(|\nabla \omer(\bx)|)\Delta \omer(\bx) -\nabla a_g(|\nabla \omer(\bx)|)\cdot \nabla \omer(\bx)\\
&=- a_g(\ber(|\bx|))\Delta \omer(\bx) - a'_g(\ber(|\bx|))\ber(|\bx|)(\ber)'(|\bx|).
\end{aligned}
$$
Therefore, since the functions~$a_g$ and~$\ber$ are positive and~$\ber$ is monotonically decreasing, also the second claim~\eqref{subhar} follows.
\end{proof}

Thus,~$\omer$ is a good prototype super-solution to~\eqref{P1e} on a certain set. However, due to the possibly non-constant prescribed boundary values~$u_0$, it must be corrected, which will be done in the next step.

\subsection{True barrier function}
Here, we correct~$\omer$ via an affine function such that it will finally give us the desired super-solution property to~\eqref{P1e}. For this purpose, let $\bk\in \mathbb{R}^d$, $c\in \mathbb{R}$, $r_0>0$ and $\delta \in(0,1)$ be arbitrary. For all $\bx \in \mathbb{R}^d \setminus B_{r_0}(\b0)$, we define
\begin{equation}\label{w2.2}
\ver(\bx)\coloneqq \omer(\bx) +\bk \cdot \bx+\bc.
\end{equation}
The key properties of the function~$\ver$ are formulated in the following lemma.

\begin{Lemma}\label{L4}
For every $K>0$ there exists a number $M>0$ depending only on~$F$ and~$K$ such that for all $\bk \in B_K(\b0)$, all $c\in \mathbb{R}$, all $\delta \in (0,1)$ and all $r_0>0$ the function~$\ver$ defined in~\eqref{w2.2} satisfies the inequalities
\begin{equation}
\label{super}
\begin{split}
-\diver \big(a(|\nabla \ver(\bx)|)\nabla \ver (\bx) \big) &\ge 0,\\
-\Delta \ver(\bx)&\ge 0
\end{split}
\end{equation}
for all $\bx \in \mathbb{R}^d \setminus B_{r_0}(\b0)$ fulfilling $\ber(|\bx|)\ge M$ with $\ber$ given by~\eqref{DFb}.
\end{Lemma}

\begin{proof}
First, it is evident that for all $\bx\in\mathbb{R}^{d}\setminus B_{r_0}(\b0)$
\begin{equation}
\nabla \ver(\bx) = \nabla \omer(\bx)+ \bk=\ber(|\bx|)\frac{\bx}{|\bx|} + \bk. \label{nabv}
\end{equation}
Consequently, a direct computation leads to
\begin{equation*}
\begin{split}
|\nabla \ver(\bx)|^2 &= (\ber)^2(|\bx|) + |\bk|^2 + 2 \ber(|\bx|) \frac{\bk \cdot \bx}{|\bx|},\\
\nabla |\nabla \ver(\bx)|&=\frac{\ber (|\bx|)(\ber)'(|\bx|)\frac{\bx}{|\bx|}+(\ber)'(|\bx|) \frac{\bx}{|\bx|}\frac{\bk \cdot \bx}{|\bx|}+\ber(|\bx|)\big( \frac{\bk}{|\bx|}-\frac{(\bk \cdot \bx) \bx}{|\bx|^3}\big) }{|\nabla \ver(\bx)|}.
\end{split}
\end{equation*}
Hence, using these identities, we obtain the following auxiliary results that will be used later
\begin{align}\label{ll1}
\nabla |\nabla \ver(\bx)| \cdot \frac{\bx}{|\bx|}&=(\ber)'(|\bx|)\frac{\ber(|\bx|)  +  \frac{\bk \cdot \bx}{|\bx|}}{|\nabla \ver(\bx)|}
\end{align}
and
\begin{equation}
\begin{split}\label{ll2}
\nabla |\nabla \ver(\bx)| \cdot \bk &=\frac{\ber(|\bx|) (\ber)'(|\bx|)\frac{\bx \cdot \bk}{|\bx|} +(\ber)'(|\bx|) \frac{(\bk \cdot \bx)^2}{|\bx|^2} + \ber(|\bx|) \big(\frac{|\bk|^2}{|\bx|}- \frac{(\bk \cdot \bx)^2}{|\bx|^3}\big)}{|\nabla \ver(\bx)|}.
\end{split}
\end{equation}

Let us now evaluate the super-solution and super-harmonicity properties. To this end, we introduce the abbreviation
\begin{align*}
L(\bx) & \coloneqq -\diver \big( a(|\nabla \ver(\bx)|)\nabla \ver(\bx) \big) \\
  & = - \nabla a(|\nabla \ver(\bx)|) \cdot \nabla \ver(\bx) - a(|\nabla \ver(\bx)|) \diver \big( \nabla \ver(\bx) \big) \eqqcolon L_1(\bx) + L_2(\bx).
\end{align*}
Employing~\eqref{nabv},~\eqref{ll1} and~\eqref{ll2}, we first calculate
\begin{align*}
L_1(\bx) & = - a'(|\nabla \ver(\bx)|) \nabla |\nabla \ver(\bx)| \cdot \nabla \ver(\bx) \\
  & = - a'(|\nabla \ver(\bx)|) \nabla |\nabla \ver(\bx)| \cdot \Big[ \ber(|\bx|) \frac{\bx}{|\bx|} + \bk \Big]\\
  & = - \frac{a'(|\nabla \ver(\bx)|)}{|\nabla \ver(\bx)|} \bigg[ \ber(|\bx|)(\ber)'(|\bx|) \Big(\ber(|\bx|)  +  \frac{\bk \cdot \bx}{|\bx|} \Big) \\
  & \hspace{0.5cm} \qquad + \ber(|\bx|) (\ber)'(|\bx|)\frac{\bx \cdot \bk}{|\bx|} +(\ber)'(|\bx|) \frac{(\bk \cdot \bx)^2}{|\bx|^2} + \ber(|\bx|) \Big(\frac{|\bk|^2}{|\bx|}- \frac{(\bk \cdot \bx)^2}{|\bx|^3}\Big) \bigg]\\
  & = \frac{a'(|\nabla \ver(\bx)|)}{|\nabla \ver(\bx)|} (\ber)'(|\bx|) \Big(|\bx|-\frac{\ber(|\bx|)}{(\ber)'(|\bx|)}\Big) \Big(\frac{|\bk|^2}{|\bx|}- \frac{(\bk \cdot \bx)^2}{|\bx|^3}\Big)\\
 &\quad-a'(|\nabla \ver(\bx)|)(\ber)'(|\bx|)|\nabla \ver(\bx)|.
\end{align*}
Next, taking into account once again~\eqref{nabv}, the relation~\eqref{omerp} and the fact that $\omer$ solves equation~\eqref{P1v}, we find
\begin{align*}
L_2(\bx)& = - a(|\nabla \ver(\bx)|)  \diver \bigg(\frac{a_g(|\nabla \omer(\bx)|)\nabla \omer(\bx)}{a_g(|\nabla \omer(\bx)|)}\bigg)\\
 & = \frac{a(|\nabla \ver(\bx)|)a'_g(|\nabla \omer(\bx)|)}{a_g(|\nabla \omer(\bx)|)} \nabla |\nabla \omer(\bx)| \cdot \nabla \omer(\bx)\\
 & = \frac{a(|\nabla \ver(\bx)|)a'_g(\ber(|\bx|))}{a_g(\ber(|\bx|))} (\ber)'(|\bx|) \ber(|\bx|).
\end{align*}
In conclusion, after a simple algebraic manipulation, we have
\begin{align*}
 L(\bx) & = \frac{a'(|\nabla \ver(\bx)|)}{|\nabla \ver(\bx)|} (\ber)'(|\bx|) \Big(|\bx|-\frac{\ber(|\bx|)}{(\ber)'(|\bx|)}\Big) \Big(\frac{|\bk|^2}{|\bx|}- \frac{(\bk \cdot \bx)^2}{|\bx|^3}\Big)\\
  & \qquad - a(|\nabla \ver(\bx)|) (\ber)'(|\bx|) \bigg(\frac{a'(|\nabla \ver(\bx)|)|\nabla \ver(\bx)|}{a(|\nabla \ver(\bx)|)}-\frac{a'_g(\ber(|\bx|))}{a_g(\ber(|\bx|))}  \ber(|\bx|)\bigg) \\
  & \eqqcolon \tilde{L}_1(\bx) +  \tilde{L}_2(\bx).
\end{align*}
We now focus on estimating the resulting terms and will show that both are non-negative in a suitably chosen set. To this end, we first relate~$\ber(|\bx|)$ and~$|\nabla \ver (\bx)|$ and provide some basic estimates, for sufficiently large values of~$\ber(|\bx|)$. Since $|\bk|\le K$, we deduce from~\eqref{nabv} that for $M_1 \coloneqq 2 K >0$ there holds
\begin{equation}
\ber(|\bx|)\ge M_1 \implies \ber(|\bx|)\le 2|\nabla \ver (\bx)|\le 4\ber(|\bx|).
\label{dist}
\end{equation}
In turn, relying on~\eqref{R3} (for both functions~$a$ and~$a_g$), we find a constant $M_2 \geq M_1$ depending only on~$F$,~$g$ and~$K$ such that
\begin{align}
\ber(|\bx|)\ge M_2 & \implies \left\{
\begin{aligned}
 (\ber)^2(|\bx|)a'_g(\ber(|\bx|))&\le - \frac12\\
 |\nabla \ver (\bx)|^2a'(|\nabla \ver (\bx)|)&\le - \frac12
\end{aligned}
\right. \nonumber \\
\label{dist2}
 & \implies a'_g(\ber(|\bx|)) < 0 \quad \textrm{and} \quad a'(|\nabla \ver (\bx)|) < 0.
\end{align}
This implication now allows us to deduce the positivity of~$\tilde{L}_1(\bx)$ and the super-harmonicity of~$\ver (\bx)$ (thus, the second claim of the lemma), provided that $\ber(|\bx|)\ge M_2$ holds. In fact, since $\ber$ is a non-negative decreasing function and by using the Cauchy-Schwarz inequality, we see that the first and second expression in large brackets in the definition of~$\tilde{L}_1(\bx)$ are non-negative. Thus, in view of~\eqref{dist2} and once again the monotonicity of~$\ber$, the sign of~$\tilde{L}_1(\bx)$ is non-negative. Secondly,~\eqref{nabv} yields $\Delta \ver(\bx) = \Delta \omer (\bx)$, thus the super-harmonicity of~$\ver (\bx)$ follows from~\eqref{subhar} and~\eqref{dist2}. In conclusion, we have the implication
\begin{equation}
\ber(|\bx|)\ge M_2 \implies - \Delta \ver(\bx) \ge 0 \quad \textrm{and} \quad \tilde{L}_1(\bx) \ge 0.
\label{dist3}
\end{equation}
Finally, we discuss the sign of~$\tilde{L}_2(\bx)$. Using~\eqref{pomer},~\eqref{pomer0},~\eqref{hlp2} and~\eqref{DA}, we evaluate
\begin{align*}
\frac{a'(s)s}{a(s)}&=\frac{F''(s)}{a(s)}-1 = \frac{sF''(s)}{1-\int_s^{\infty} F''(t)\dt}-1, \\
\frac{a'_g(s)s}{a_g(s)}&=\frac{sF''_g(s)}{F'_g(s)}-1 = \frac{s~g(s)F''(s)}{1-\int_s^{\infty} g(t)F''(t)\dt}-1.
\end{align*}
In this way the expression for~$\tilde{L}_2(\bx)$ reduces to
\begin{multline}
\label{est-L2_tilde}
\tilde{L}_2(\bx) = - a(|\nabla \ver(\bx)|) (\ber)'(|\bx|)\\
  \times \bigg(\frac{|\nabla \ver(\bx)|F''(|\nabla \ver(\bx)|)}{1-\int_{|\nabla \ver(\bx)|}^{\infty} F''(t)\dt} -\frac{\ber(|\bx|)g(\ber(|\bx|))F''(\ber(|\bx|))}{1-\int_{\ber(|\bx|)}^{\infty} g(t)F''(t)\dt}\bigg).
\end{multline}
For $\ber(|\bx|)\ge M_2$, we then find, employing~\eqref{dist}, the oscillation assumption~\eqref{A1} on~$F$ and the fact that~$g$ is a positive and monotonically decreasing function, the inequalities
\begin{equation*}
 |\nabla \ver(\bx) F''(|\nabla \ver(\bx)|) \geq \frac{1}{2C_2} \ber(|\bx|) F''(\ber(|\bx|))
\end{equation*}
and
\begin{equation*}
 \int_{|\nabla \ver(\bx)|}^{\infty} F''(t)\dt \geq \frac{1}{2C_2 g(\ber(|\bx|))} \int_{\ber(|\bx|)}^{\infty} g(t) F''(t)\dt.
\end{equation*}
At this stage, we select a number $M \geq M_2$ depending only on~$F$,~$g$ and~$K$ such that $2 C_2 g(M) \leq 1$ holds which is possible because of $g(t)\to 0$ as $t\to \infty$, according to Lemma~\ref{L2}. With this choice and the integrability of ~$F''$  over~$\mathbb{R}^+$ (with integral equal to~$1$), we see easily that the expression in large brackets on the right-hand side of~\eqref{est-L2_tilde} is non-negative, whenever $\ber(|\bx|)\ge M$ holds. Consequently, using also the facts that $\ber$ is monotonically decreasing and that~$a$ is non-negative (thanks to~\eqref{min2} and the non-negativity of~$F'$), we arrive at the implication
\begin{equation*}
\ber(|\bx|)\ge M \implies\tilde{L}_2(\bx) \ge 0.
\end{equation*}
Combined with~\eqref{dist3}, we finally conclude that~$L(\bx)\ge0$ holds for all $\bx \in \mathbb{R}^d \setminus B_{r_0}(\b0)$ with $\ber(|\bx|)\ge M$, and the proof of the lemma is complete.
\end{proof}

\subsection{Estimates of the normal derivatives}
Once the true barrier function from Lemma~\ref{L4} is at our disposal, we can return to study the normal derivative, with the aim to prove an estimate of the form~\eqref{basis2}.

The strategy of proof is as follows: by adjusting the true barrier function from Lemma~\ref{L4} to our needs, we construct in a first step a \emph{\textup{(}local upper\textup{)} barrier function} relative to the Dirichlet problem~\eqref{P1} for an arbitrary given boundary point~$\bx_0$. This means that we specify a relative neighborhood~$U(\bx_0)$ of~$\bx_0$ in~$\Omega$ and a local Lipschitz continuous function~$v$ (which will be an affine perturbation of the function from Lemma~\ref{L4}) defined on~$U(\bx_0)$ such that
\begin{enumerate}[font=\normalfont, label=(\roman{*}), ref=(\roman{*})]
 \item $v$ is a super-solution of the equation in~$U(\bx_0)$, i.e., $-\diver (a(|\nabla v|)\nabla v ) \geq 0$ in~$U(\bx_0)$,
 \item $v$ lies above the solution $u$ on~$\partial U(\bx_0)$ and coincides for $\bx_0$, i.e. $v \geq u$ on~$\partial U(\bx_0)$ and $v(\bx_0)=u(\bx_0)$.
\end{enumerate}
Via a comparison principle applied to the solution~$u$ and the super-solution~$v$, we can finally estimate the normal derivative of~$u$ at~$\bx_0$ by the sup-norm of the derivative of the barrier function~$v$ (which in turn is bounded in terms of the data) and arrive at the assertion~\eqref{basis2}.

\smallskip

\noindent
\begin{minipage}[h]{0.4\linewidth}
 \includegraphics[width=\textwidth]{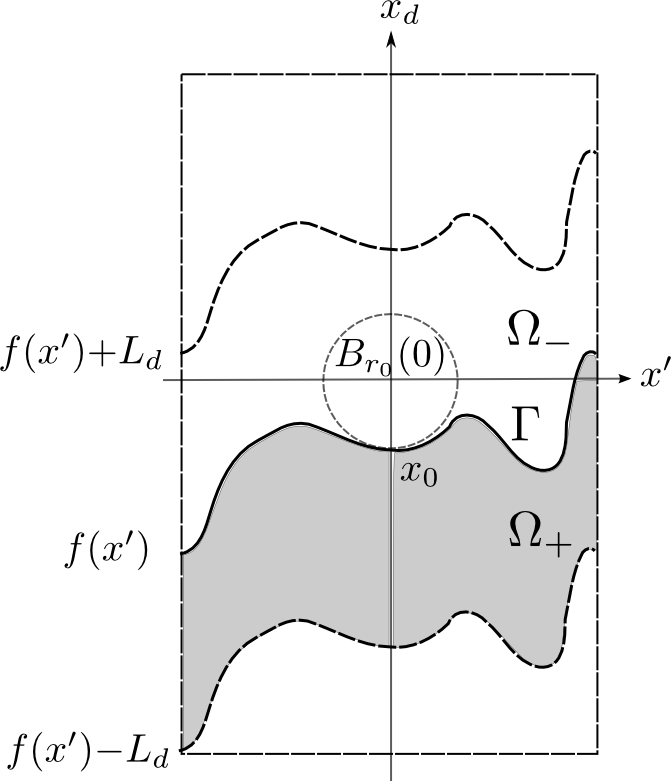}
\end{minipage}
\begin{minipage}[h]{0.59\linewidth}
Now we start with the derivation of the estimates for the normal derivative. Since $\Omega$ is by assumption of class~$\mathcal{C}^{1}$ and satisfies an exterior ball condition, we find positive constants $r_0$, $L$, $L_d$ and~$N$ depending only on~$\Omega$ such that we can suppose that an arbitrary boundary point $\bx_0\in \partial \Omega$  is given, after an orthogonal transformation, by $\bx_0=(\b0, -r_0)$ (we use the notation $\bx=(\bx',x_d)$) and that we have the inclusions
\begin{equation*}
\begin{aligned}
\Gamma & \coloneqq \{\bx\in \mathbb{R}^d \colon |\bx'|< L, \; f(\bx')=x_d\} \\
       & \subset \partial \Omega,\\
\Omega_+ & \coloneqq \{\bx\in \mathbb{R}^d \colon |\bx'|< L, \; f(\bx') - L_d < x_d < f(\bx')\} \\
       &\subset \Omega,\\
\Omega_{-} & \coloneqq \{\bx\in \mathbb{R}^d \colon |\bx'|< L, \; f(\bx')<x_d<f(\bx')+L_d\} \\
       & \subset \mathbb{R}^d \setminus\Omega,
\end{aligned}
\end{equation*}
with a function $f\in \mathcal{C}^{1}(-L,L)^{d-1}$ fulfilling $\|f\|_{1,\infty}\le N$, $f(\b0')=-r_0$ and $D_i f(\b0)=0$ for all $i=1,\ldots, d-1$.
\end{minipage}

In addition, we may suppose that~$r_0$ is so small that $B_{r_0}(\b0)\subset \Omega_{-}$ holds and that
\begin{equation}
\label{bound-dist-cond}
M^*(|\bx|-r_0)\ge |\bx-\bx_0|^2
\end{equation}
is satisfied for all $\bx \in \Gamma$, for some constant~$M^*$ depending only on~$\Omega$ and~$r_0$ (this can for example be seen easily if also $B_{2r_0}(\b0',r_0) \subset \Omega_{-}$ holds).

In this setting, for an arbitrary $\delta\in(0,1)$ (to be specified later on) we can work with the functions~$\ber$ and~$\omer$ introduced in~\eqref{DFb} and~\eqref{dfomer}, and with the function~$\ver$ from~\eqref{w2.2} for the specific choices $\bk \coloneqq \nabla u_0(\bx_0)$ and $c=u_0(\bx_0) - \nabla u_0(\bx_0) \cdot \bx_0$, that is, with
\begin{equation}\label{dfvef}
\ver(\bx)=\omer(\bx) +\nabla u_0 (\bx_0) \cdot (\bx-\bx_0) + u_0(\bx_0).
\end{equation}
Note that $\ver$ is well defined outside the ball $B_{r_0}(\b0)$ and so it is well-defined also in $\Omega_+$. In addition, it is clear that $|\bk|\le \|\nabla u_0\|_{\infty}$ holds, hence, we can choose $K \coloneqq \|\nabla u_0\|_{\infty}$ and fix the number~$M$ (depending only on~$F$ and this~$K$) according to Lemma~\ref{L4}. Furthermore, since $(F'_g)^{-1}$ maps $[0,1)$ to $[0,\infty)$ and is monotonically increasing, we can fix a number $\delta_{\max} \in (0,1/2)$ such that
\begin{equation}\label{C1}
(F'_g)^{-1}(s)\ge \max\big\{M, M^* \|u_0\|_{1,\infty}\big\}  \quad \textrm{ for all } s \in \big[(1-2\delta_{\max})^{d-1},1\big).
\end{equation}
From now on, we will consider arbitrary $\delta \in (0,\delta_{\max})$. Then, from~\eqref{DFb} and~\eqref{C1} it follows that
\begin{equation}\label{DFb2}
r_0 < |\bx| \leq \frac{(1-\delta_{\max})r_0}{1-2\delta_{\max}} \eqqcolon r_{\max}  \implies \ber(|\bx|) \geq M.
\end{equation}
Consequently, using Lemma~\ref{L4}, we see that $\ver$ is a super-solution to~\eqref{P1e} in the set $(B_{r_{\max}}\setminus B_{r_0})\cap \Omega_+$, which is the first crucial property of an upper barrier.

Next, we want to identify a part of $\Gamma$ on which $u(\bx) = u_0(\bx)\le\ver(\bx)$ holds, that is, where
\begin{equation}\label{dui}
\omer(\bx) +\nabla u_0(\bx_0)\cdot (\bx -\bx_0) +u_0(\bx_0)-u_0(\bx) \geq 0.
\end{equation}
From Taylor expansion of~$u_0$ and the~$\mathcal{C}^{1,1}$-regularity assumption on~$u_0$ we know that
$$
\begin{aligned}
&\left|u_0(\bx)-u_0(\bx_0)-\nabla u_0(\bx_0)\cdot (\bx - \bx_0)\right|\le \|u_0\|_{1,\infty} |\bx-\bx_0|^2,
\end{aligned}
$$
so to verify~\eqref{dui} it is enough to check where
\begin{equation}
\omer(\bx) - \|u_0\|_{1,\infty} |\bx - \bx_0|^2 \geq 0 \label{dui3}
\end{equation}
holds. Using the definitions of $\ber$ in~\eqref{DFb} and of~$\omer$ in~\eqref{dfomer}, combined with the fact that~$(F'_g)^{-1}$ is monotonically increasing, we have for all $\bx\in \Gamma$
\begin{equation*}
\omer(\bx)\ge (|\bx|-r_0) (F'_g)^{-1} \bigg(\frac{(1-\delta)^{d-1}r^{d-1}_0}{|\bx|^{d-1}}\bigg).
\end{equation*}
Consequently, in order to guarantee~\eqref{dui3} and thus~\eqref{dui} it is sufficient, in view of~\eqref{bound-dist-cond}, to have
\begin{equation*}
(F'_g)^{-1} \bigg(\frac{(1-\delta)^{d-1}r^{d-1}_0}{|\bx|^{d-1}}\bigg)\ge M^* \|u_0\|_{1,\infty},
\end{equation*}
which is indeed true for all~$\bx$ with $r_0 \leq |\bx| \leq r_{\max}$, by the choices of the parameter~$\delta_{\max}$ in~\eqref{DFb2} and of the radius~$r_{\max}$ in~\eqref{DFb2}. Thus, we have verified
\begin{equation}
\label{boundary_values_Gamma}
u(\bx)\le\ver(\bx) \qquad \textrm{for all } \bx \in \Gamma \textrm{ with } r_0 \leq |\bx|\leq r_{\max}.
\end{equation}

Finally, in order to complete the second property of the barrier function that it lies above the solution on all of the boundary of a relative neighborhood of~$\bx_0$, we still need to take care of the values of $\ver(\bx)$ inside of $\Omega$ (but close to~$\bx_0$). This shall now be accomplished by a suitable choices of a local neighborhood and of~$\delta \in (0,\delta_{\max})$.

\noindent
\begin{minipage}[h]{0.4\linewidth}
 \includegraphics[width=\textwidth]{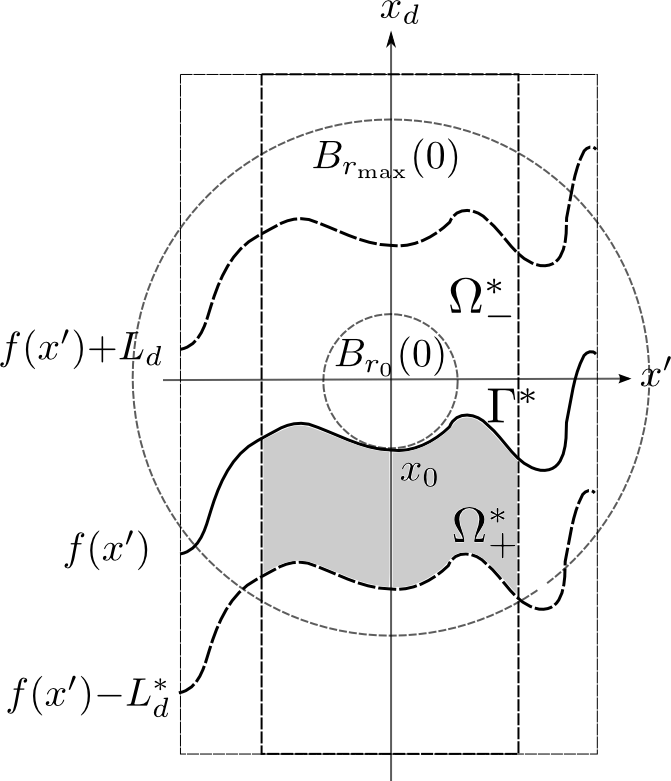}
\end{minipage}
\begin{minipage}[h]{0.59\linewidth}
First, since~$r_0$ and $r_{\max}$ are already fixed (in dependence on~$\Omega$,~$F$ and~$u_0$), we can zoom in the neighborhood of~$\bx_0$ and find~$L^*$ and $L^*_{d} \leq L_d$ sufficiently small (depending again only on data) such that
\begin{equation*}
\begin{aligned}
\Gamma^*& \coloneqq \{\bx\in \mathbb{R}^d \colon |\bx'|< L^*, \; f(\bx')=x_d\} \\
    & \subset \partial \Omega,\\
\Omega^*_+& \coloneqq \{\bx\in \mathbb{R}^d \colon |\bx'|< L^*, \; f(\bx') - L^*_d < x_d < f(\bx'\} \\
    & \subset \Omega \cap (B_{r_{\max}}\setminus B_{r_0}),\\
\Omega^*_{-}& \coloneqq \{\bx\in \mathbb{R}^d \colon |\bx'|< L^*,  \; f(\bx')<x_d<f(\bx')+L_d\} \\
    & \subset \mathbb{R}^d \setminus\Omega.
\end{aligned}
\end{equation*}
By these choices, due to~\eqref{DFb2} and~\eqref{boundary_values_Gamma}, $\ver$ is a super-solution to~\eqref{P1e} in the relative neighborhood~$\Omega^*_+$ of~$\bx_0$ and satisfies $\ver \geq u$ on~$\Gamma^*$, for all $\delta < \delta_{\max}$. Our goal is to show $\ver\ge u$ on the remaining part of $\partial \Omega_+^*$, and this is indeed the point, where we shall use the assumption~\eqref{A2}.
\end{minipage}

In view of the choice of~$r_0$ and~$\Omega_+^*$ there exists $\eta>0$ independent of $\delta$ such that for all $\bx \in \partial \Omega_+^* \setminus \Gamma^*$ there holds
$$
|\bx|\ge r_0 + \eta.
$$
Hence, from the definition of~$\ver$ in~\eqref{dfvef} and of~$\omer$ in~\eqref{dfomer} we find
\begin{equation}
\ver(\bx) \ge \int_{r_0}^{r_0 +\eta} \ber (r) \dr - C^*(\Omega)  \|u_0\|_{1,\infty}. \label{nemohu}
\end{equation}
On the other hand, we know $\|u\|_{\infty} \le \|u_0\|_{\infty}$ from~\eqref{AE2}. Therefore, in order to show that $\ver \ge u$ on $\partial \Omega_+^*$, it is enough to verify that we can choose $\delta\in (0,\delta_{\max})$ in such way that
\begin{equation}
\int_{r_0}^{r_0 +\eta} \ber (r) \dr \ge C^*(\Omega)  \|u_0\|_{1,\infty} + \|u_0\|_{\infty}.\label{konec}
\end{equation}
Using the definition of~$\ber$ in~\eqref{DFb} and the substitution formula, we deduce that
\begin{align*}
\int_{r_0}^{r_0 +\eta} \ber (r) \dr &=\int_{r_0}^{r_0 + \eta}(F'_g)^{-1} \bigg(\frac{(1-\delta)^{d-1}r^{d-1}_0}{r^{d-1}}\bigg)\dr\\
&=\frac{(1-\delta)r_0}{d-1} \int^{(1-\delta)^{d-1}}_{\frac{(1-\delta)^{d-1}r^{d-1}_0}{(r_0+\eta)^{d-1}}
}  \frac{(F'_g)^{-1}(s)}{s^{\frac{d}{d-1}}} \ds.
\end{align*}
If we now introduce
\begin{equation*}
 \alpha \coloneqq \min \Big\{ \frac{r_0}{d-1}, 1 - \frac{r_0^{d-1}}{(r_0 + \eta)^{d-1}}\Big\}
\end{equation*}
(depending only on~$\Omega$, as~$r_0$ and~$\eta$ are already fixed), the above integral can be estimated by
\begin{align*}
\int_{r_0}^{r_0 +\eta} \ber (r)\dr & \ge \alpha \int^{(1-\delta)^{d-1}}_{1-\alpha}  (F'_g)^{-1}(s)\ds\\
&= \alpha\int_{(F'_g)^{-1}(1-\alpha)}^{(F'_g)^{-1}((1-\delta)^{d-1})} t F''_g(t)\dt
\end{align*}
(note that the integral on the right-hand side is negative whenever $1-\alpha > (1-\delta)^{d-1}$ holds, hence, the previous inequality is trivially satisfied in this case). Thanks to Lemma~\ref{L2} (more precisely, by~\eqref{A3} and~\eqref{A2} for~$F_g$), we know first that $(F'_g(s))^{-1} \to \infty$ as $s\to 1$ and secondly that
\begin{equation*}
\int_1^{\infty} t F''_g(t)\dt = \infty.
\end{equation*}
Therefore, we can fix $\delta\in (0, \delta_{\max})$ (depending only on~$\Omega$,~$F$ and $u_0$) such that
\begin{equation*}
\int_{r_0}^{r_0 +\eta} \ber (r)\dr \geq \alpha\int_{(F'_g)^{-1}(1-\alpha)}^{(F'_g)^{-1}((1-\delta)^{d-1})} t F''_g(t) \dt \ge  C^*(\Omega)  \|u_0\|_{1,\infty} + \|u_0\|_{\infty}
\end{equation*}
holds, which yields the desired inequality~\eqref{konec}. Therefore, we have proved
\begin{equation}
\label{boundary_values_final}
 \ver \ge u \qquad \textrm{on } \partial \Omega_+^*
\end{equation}
and have finished the construction of the barrier function on the relative neighborhood~$\Omega_+$ of~$\bx_0$.

It now remains to establish the bound~\eqref{basis2} for the normal derivative of~$u$, locally at~$\bx_0$. For this purpose, we recall that~$\ver$ is a super-solution to~\eqref{P1e} and $u$ is a solution to~\eqref{P1e} in $\Omega_+^*$, Thus, in view of~\eqref{boundary_values_final} we obtain from the comparison principle for elliptic equations in divergence form
\begin{equation*}
\ver \ge u \qquad \textrm{ in } \Omega_+^*.
\end{equation*}
This now allows to estimate the normal derivative. Indeed, for any $0<h\le L_d^*$ we have $\bx\coloneqq (\b0',-r_0-h)\in \Omega_+^*$, and therefore, taking into account also $\ver(\bx_0) = u(\bx_0)$, we find the following estimate
\begin{align*}
\frac{u(\bx_0)-u(\bx)}{h} &= \frac{\ver(\bx_0)-\ver(\bx)}{h} + \frac{\ver(\bx)-u(\bx)}{h} \\
 & \ge \frac{\ver(\bx_0)-\ver(\bx)}{h} \ge -\|\ver\|_{1,\infty}\ge -C(\Omega, F, u_0).
\end{align*}
Thus, recalling that the outer unit normal to~$\partial \Omega$ in~$\bx_0$ is given by $e_n$, we obtain in the passage $h\to 0_+$ the lower bound
\begin{equation*}
\frac{\partial u(\bx_0)}{\partial \bn} \ge -C(\Omega, F, u_0).
\end{equation*}
Repeating the whole procedure with~$\omer$ replaced by~$-\omer$, we get the opposite inequality
\begin{equation*}
\frac{\partial u(\bx_0)}{\partial \bn} \le C(\Omega, F, u_0).
\end{equation*}
The latter two inequalities imply~\eqref{basis2}, which in turn, due to Subsection~\ref{ss31}, provides
\begin{equation*}
\|\nabla u\|_{\infty} \le C(\Omega, F, u_0)
\end{equation*}
(which, as the index~$\varepsilon$ was dropped, is precisely the uniform bound~\eqref{basis}). This finishes the proof of the theorem. \qed

\bibliographystyle{plain}

\end{document}